\documentclass[11pt]{amsart} \textwidth=14.5cm \oddsidemargin=1cm
\usepackage{amsmath,wasysym}
\usepackage{amsxtra}
\usepackage{amscd}
\usepackage{amsthm}
\usepackage{amsfonts}
\usepackage{amssymb}
\usepackage{eucal}
\usepackage{graphics, color}
\usepackage{mathrsfs}
\usepackage[usenames,dvipsnames]{xcolor}
\usepackage{tikz}
\usepackage{stmaryrd} %\llbracket, \rrbracket
\usepackage[pdfusetitle,hidelinks,bookmarksdepth=2]{hyperref}

\input prepictex
\input pictex
\input postpictex
\newtheorem{thm}{Theorem} [section]
\newtheorem{lem}[thm]{Lemma}
\newtheorem{cor}[thm]{Corollary}
\newtheorem{prop}[thm]{Proposition}

\theoremstyle{definition}
\newtheorem{definition}[thm]{Definition}
\newtheorem{example}[thm]{Example}

\theoremstyle{remark}
\newtheorem{rem}[thm]{Remark}

\numberwithin{equation}{section}

\newcommand{\thmref}[1]{Theorem~\ref{#1}}
\newcommand{\secref}[1]{Section~\ref{#1}}
\newcommand{\lemref}[1]{Lemma~\ref{#1}}

\newcommand{\corref}[1]{Corollary~\ref{#1}}

\newcommand{\defref}[1]{Definition~\ref{#1}}

\newcommand{\N}{\mathbb{N}}
\newcommand{\llb}{\llbracket}
\newcommand{\rrb}{\rrbracket}

\tikzstyle{elem} = [draw=none, fill=none]

\advance\headheight by 2pt

\begin{document}

\title[Chain Posets]{Chain Posets}
\author{Ian T.\ Johnson}
\address{Department of Mathematics, University of Virginia, Charlottesville, VA 22904, USA}
\email{ij6fd@virginia.edu}

\begin{abstract}
  A chain poset, by definition, consists of chains of ordered elements in a
  poset.  We study the chain posets associated to two posets: the Boolean
  algebra and the poset of isotropic flags.  We prove that, in both cases, the
  chain posets satisfy the strong Sperner property and are rank-log concave.
\end{abstract}

% \subjclass[2010]{17B67}

\maketitle

%\setcounter{tocdepth}{1}
%\tableofcontents

\section{Introduction}
\label{sec:intro}
Given a poset $(P, \leq)$, it is often useful to consider \emph{chains} in $P$,
or ordered collections of elements of $P$.  In this paper, we consider a poset
structure on the set of chains in a poset $P$, formulated in
\defref{def:kflags}.

While gradedness is preserved by this chain poset structure (that is, if $P$ is
a graded poset then so is $P[k]$ for all $k\in\N$), other properties like
rank-symmetry and rank-unimodality are not, in general. This paper focuses
primarily on the special cases of the well-known Boolean algebra $B_n$ and a
close relative, the poset of isotropic flags $I_n$, for which stronger results
can be proved regarding their chain posets. In particular, for all $k\in\N$,
$B_n[k]$ and $I_n[k]$ are both rank-log concave, which implies
rank-unimodality, and satisfy the strong Sperner property.

The key to proving these results is the expression of $B_n$ and $I_n$ as direct
products of simpler posets; for these simpler posets, it is almost trivial to
prove the desired properties. The chain poset structure is compatible with the
direct product; that is, for posets $P$ and $Q$ and $k\in\N$,
$(P\times Q)[k] \cong P[k]\times Q[k]$. Finally, a result of~\cite{Engel} on
the direct products of posets allows us to prove the desired properties of
$B_n[k]$ and $I_n[k]$.

In the following section, we will expand this brief sketch by proving the
claims made at each step and explaining their precise formulations in more
detail.

\vspace{.3cm}

\textbf{Acknowledgement.} This undergraduate research was mentored by Professor
Weiqiang Wang of the University of Virginia and was partially supported by his
NSF grants DMS-1405131 and DMS-1702254.  The author would like to thank his
mentor for his help in preparing this paper and for posing the questions which
it answers.

\section{Properties of chain posets}
\subsection{Basic properties}
We define the notion of \textit{chain posets} as follows.

\begin{definition}
  \label{def:kflags}
  Let $(P, \leq)$ be a poset. Given $k \in \N$, define the poset of $k$-chains
  \[
    P[k] = \{(x_1 \leq x_2 \leq \dotsb \leq x_k) \mid x_i \in P\ \forall i\}.
  \]
  The poset structure $(P[k], \leq_k)$ is defined by
  \[
    (x_1 \leq x_2 \leq \dotsb \leq x_k) \leq_k (x'_1 \leq x'_2 \leq \dotsb \leq x'_k) \Leftrightarrow x_i \leq x'_i\ \forall i.
  \]
\end{definition}

For a graded poset $P$ (we omit the ordering $\leq$ when it is clear from
context), we use the convention that the rank function $\rho$ satisfies
$\rho(x) = 0$ for all minimal $x \in P$ and $\rho(x) = \rho(y) + 1$ for all
$x,y\in P$ such that $x \gtrdot y$.

\begin{prop}
  \label{prop:graded}
  Let $P$ be a graded poset of rank $N$; that is, every maximal chain in $P$
  has $N$ elements. Then $P[k]$ is graded of rank $kN$. Furthermore, we have
  the following expression for the rank $\rho$ of an element in $P[k]$:
  \[
    \rho(x_1 \leq x_2 \leq \dotsb \leq x_k) = \sum_{i=1}^k \rho(x_i).
  \]
\end{prop}
\begin{proof}
  This follows from a simple observation regarding the elements covered by a
  particular chain: given $x = (x_1 \leq x_2 \leq \dotsb \leq x_k)$,
  $y = (y_1 \leq y_2 \leq \dotsb \leq y_k) \lessdot x$ if and only if
  $y_i = x_i$ for all $1 \leq i \leq k$ with exactly one exception $j$, and
  $y_j \lessdot x_j$. It is clear that we indeed have $y < x$, and by our
  condition that $y_j \lessdot x_j$ for exactly one $j$ we see that there can
  be no element $z$ such that $y < z < x$.

  Since $P$ is graded, this implies that if $y \lessdot x$, then
  \[
    \sum_{i=1}^k\rho(y_i) = \sum_{i=1}^k\rho(x_i) - 1.
  \]
  Thus, the desired statement follows by induction on $\sum_{i=1}^k\rho(x_i)$.
\end{proof}

It is natural to ask whether other common properties of posets, such as
rank-symmetry or rank-unimodality, carry over to their posets of $k$-chains in
a similar manner. In general, this is not the case, as shown by the following
counterexamples.

\begin{example}
  \label{ex:countersym}
  The following poset $P$ (left) is rank-symmetric, but the corresponding
  $P[2]$ (right) is not. In the Hasse diagram for $P[2]$, we abbreviate the
  chain $C \leq A$ by $CA$, and so on.
  \\

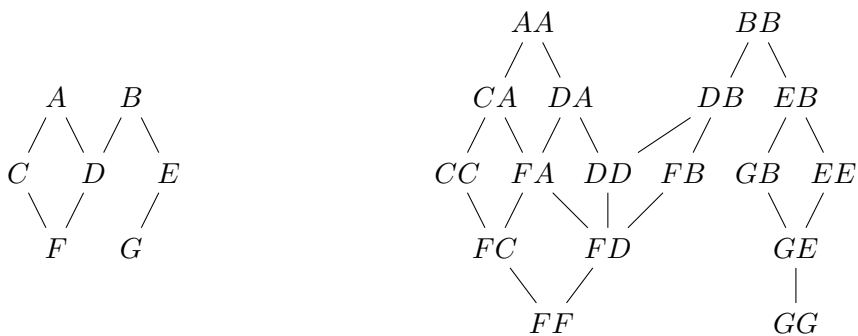
  \begin{minipage}{0.5\textwidth}
    \begin{center}
      \begin{tikzpicture}
        \node[elem] (A) at (-0.5, 1) {$A$};
        \node[elem] (B) at (0.5, 1) {$B$};

        \node[elem] (C) at (-1, 0) {$C$};
        \node[elem] (D) at (0, 0) {$D$};
        \node[elem] (E) at (1, 0) {$E$};

        \node[elem] (F) at (-0.5, -1) {$F$};
        \node[elem] (G) at (0.5, -1) {$G$};

        \path[draw]
        (A) edge node {} (C)
        edge node {} (D)
        (B) edge node {} (D)
        edge node {} (E)

        (C) edge node {} (F)
        (D) edge node {} (F)
        (E) edge node {} (G);
      \end{tikzpicture}
    \end{center}
  \end{minipage}
  \begin{minipage}{0.5\textwidth}
    \begin{center}
      \begin{tikzpicture}
        \node[elem] (AA) at (-1.5, 2) {$AA$};
        \node[elem] (BB) at (1.5, 2) {$BB$};

        \node[elem] (CA) at (-2, 1) {$CA$};
        \node[elem] (DA) at (-1, 1) {$DA$};
        \node[elem] (DB) at (1, 1) {$DB$};
        \node[elem] (EB) at (2, 1) {$EB$};

        \node[elem] (CC) at (-2.5, 0) {$CC$};
        \node[elem] (FA) at (-1.5, 0) {$FA$};
        \node[elem] (DD) at (-0.5, 0) {$DD$};
        \node[elem] (FB) at (0.5, 0) {$FB$};
        \node[elem] (GB) at (1.5, 0) {$GB$};
        \node[elem] (EE) at (2.5, 0) {$EE$};

        \node[elem] (FC) at (-2, -1) {$FC$};
        \node[elem] (FD) at (-0.5, -1) {$FD$};
        \node[elem] (GE) at (2, -1) {$GE$};

        \node[elem] (FF) at (-2.5/2, -2) {$FF$};
        \node[elem] (GG) at (2, -2) {$GG$};

        \path[draw]
        (AA) edge node {} (CA)
        edge node {} (DA)
        (BB) edge node {} (DB)
        edge node {} (EB)

        (CA) edge node {} (CC)
        edge node {} (FA)
        (DA) edge node {} (DD)
        edge node {} (FA)
        (DB) edge node {} (DD)
        edge node {} (FB)
        (EB) edge node {} (EE)
        edge node {} (GB)

        (CC) edge node {} (FC)
        (FA) edge node {} (FC)
        edge node {} (FD)
        (DD) edge node {} (FD)
        (FB) edge node {} (FD)
        (GB) edge node {} (GE)
        (EE) edge node {} (GE)

        (FC) edge node {} (FF)
        (FD) edge node {} (FF)
        (GE) edge node {} (GG);
      \end{tikzpicture}
    \end{center}
  \end{minipage}
\end{example}

\begin{example}
  \label{ex:counteruni}
  The following poset $P$ (left) is rank-unimodal, but the corresponding $P[2]$
  (right) is not. In the Hasse diagram for $P[2]$, we abbreviate the chain
  $C \leq A$ by $CA$, and so on.
  \\

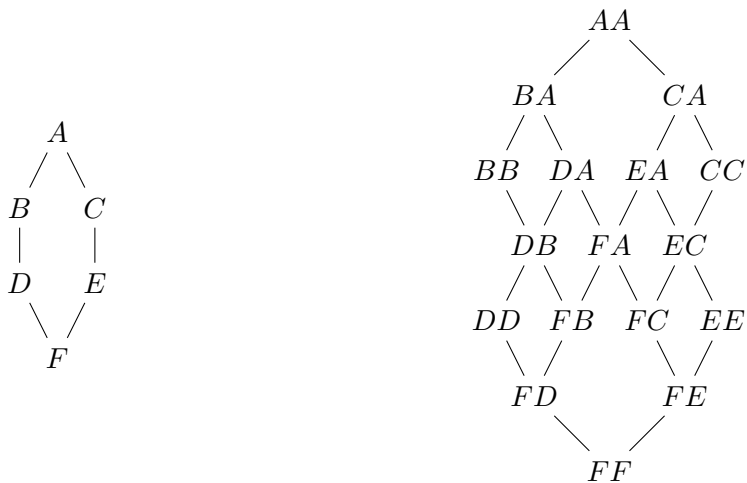
  \begin{minipage}{0.5\textwidth}
    \begin{center}
      \begin{tikzpicture}
        \node[elem] (A) at (0, 3) {$A$};

        \node[elem] (B) at (-0.5, 2) {$B$};
        \node[elem] (C) at (0.5, 2) {$C$};

        \node[elem] (D) at (-0.5, 1) {$D$};
        \node[elem] (E) at (0.5, 1) {$E$};

        \node[elem] (F) at (0, 0) {$F$};

        \path[draw]
        (A) edge node {} (B)
        edge node {} (C)

        (B) edge node {} (D)
        (C) edge node {} (E)

        (D) edge node {} (F)
        (E) edge node {} (F);
      \end{tikzpicture}
    \end{center}
  \end{minipage}
  \begin{minipage}{0.5\textwidth}
    \begin{center}
      \begin{tikzpicture}
        \node[elem] (AA) at (0, 6) {$AA$};

        \node[elem] (BA) at (-1, 5) {$BA$};
        \node[elem] (CA) at (1, 5) {$CA$};

        \node[elem] (BB) at (-1.5, 4) {$BB$};
        \node[elem] (DA) at (-0.5, 4) {$DA$};
        \node[elem] (EA) at (0.5, 4) {$EA$};
        \node[elem] (CC) at (1.5, 4) {$CC$};

        \node[elem] (DB) at (-1, 3) {$DB$};
        \node[elem] (FA) at (0, 3) {$FA$};
        \node[elem] (EC) at (1, 3) {$EC$};

        \node[elem] (DD) at (-1.5, 2) {$DD$};
        \node[elem] (FB) at (-0.5, 2) {$FB$};
        \node[elem] (FC) at (0.5, 2) {$FC$};
        \node[elem] (EE) at (1.5, 2) {$EE$};

        \node[elem] (FD) at (-1, 1) {$FD$};
        \node[elem] (FE) at (1, 1) {$FE$};

        \node[elem] (FF) at (0, 0) {$FF$};

        \path[draw]
        (AA) edge node {} (BA)
        edge node {} (CA)

        (BA) edge node {} (BB)
        edge node {} (DA)
        (CA) edge node {} (EA)
        edge node {} (CC)

        (BB) edge node {} (DB)
        (DA) edge node {} (DB)
        edge node {} (FA)
        (EA) edge node {} (FA)
        edge node {} (EC)
        (CC) edge node {} (EC)

        (DB) edge node {} (DD)
        edge node {} (FB)
        (FA) edge node {} (FB)
        edge node {} (FC)
        (EC) edge node {} (FC)
        edge node {} (EE)

        (DD) edge node {} (FD)
        (FB) edge node {} (FD)
        (FC) edge node {} (FE)
        (EE) edge node {} (FE)

        (FD) edge node {} (FF)
        (FE) edge node {} (FF);
      \end{tikzpicture}
    \end{center}
  \end{minipage}
\end{example}

\subsection{Direct products}
Recall that, given two posets $(P, \leq_P)$ and $(Q, \leq_Q)$, their
\emph{direct product} is defined as $(P \times Q, \leq)$, where the set
$P \times Q$ is the ordinary Cartesian product of the sets $P$ and $Q$ and the
relation $\leq$ is defined by
\[
  (p_1, q_1) \leq (p_2, q_2) \Leftrightarrow
  (p_1 \leq_P p_2) \wedge (q_1 \leq_Q q_2).
\]

This direct product is similar in many ways to the definition of the chain
posets in~\secref{sec:intro}. Indeed, the definition of the ordering $\leq$ is
identical, motivating the following lemma and its corollary.

\begin{lem}
  Let $(P, \leq_P)$ and $(Q, \leq_Q)$ be posets. Then for any $k \in \N$,
  \[
    (P \times Q)[k] \cong P[k] \times Q[k].
  \]
\end{lem}
\begin{proof}
  We use the following correspondence between the elements of $(P \times Q)[k]$
  and $P[k] \times Q[k]$:
  \[
    (p_1, q_1) \leq (p_2, q_2) \leq \dotsb \leq (p_k, q_k) \leftrightarrow
    (p_1 \leq_P p_2 \leq_P \dotsb \leq_P p_k, q_1 \leq_Q q_2 \leq_Q \dotsb
    \leq_Q q_k).
  \]
  That this relation is bijective follows immediately, as does the required
  ``order-preservation'' property: if $e_1, e_2 \in (P \times Q)[k]$ and
  $e'_1, e'_2 \in P[k] \times Q[k]$ are the corresponding elements, then we
  have $e_1 \leq e_2 \Leftrightarrow e'_1 \leq e'_2$.
\end{proof}

\begin{cor}
  \label{cor:dirprod}
  Let $P$ be a poset and $k \in \N$. Then for any $n \in \N$, we have
  \[
    (P^n)[k] \cong {(P[k])}^n,
  \]
  where
  $P^n \equiv \underbrace{P \times P \times \dotsb \times P}_{\text{n times}}$.
\end{cor}

\cite[Chapters~4.5 and~4.6]{Engel} provide a very useful method for proving
that a given direct product is rank-unimodal and Sperner, provided that its
``factors'' satisfy certain properties.

The first of these properties is \emph{normality}: if $P$ is a graded poset and
$i \geq 0$, let $P_i \equiv \{x\in P \mid \rho(x) = i\}$. Additionally, for any
subset $A \subseteq P$, define $\nabla(A)$ to be the set of all elements of $P$
which cover some element of $A$. Then, following~\cite{Engel}, we say that $P$
is normal if
\[
  \frac{|A|}{|P_i|} \leq \frac{|\nabla(A)|}{|P_{i+1}|}
\]
for all $A \subseteq P_i$ and $i = 0, \dotsc, n-1$. By
\cite[Corollary~4.5.3]{Engel} normality implies the strong Sperner property:
for $k \in \N$, a graded poset $P$ has the \emph{$k$-Sperner property} if no
union of $k$ antichains of $P$ contains more elements than the union of the $k$
largest levels of $P$; $P$ has the \emph{strong Sperner property} if it has the
$k$-Sperner property for all $k \in \N$.

The second of these properties is \emph{rank-log concavity}: if $P$ is a graded
poset, we say that $P$ is \emph{rank-log concave} if the sequence
$|P_0|, |P_1|, \dotsc, |P_n|$ is log concave, that is, if
$|P_i|^2 \geq |P_{i-1}||P_{i+1}|$ for all $i = 1,2,\dotsc,n-1$. Noting that
$|P_i| > 0$ for all $0 \leq i \leq n$, we can
use~\cite[Proposition~5.11]{Stanley} to conclude that rank-log concavity
implies rank-unimodality.

\begin{thm}[{\cite[Theorem~4.6.2]{Engel}}]
  \label{thm:product}
  If $P$ and $Q$ are posets which are both normal and rank-log concave, then
  their direct product $P \times Q$ is also normal and rank-log concave.
\end{thm}

\subsection{The Boolean algebra}
\label{subsec:boolean}
We are now ready to introduce the Boolean algebra and investigate its
corresponding chain posets. Recall that the Boolean algebra $B_n$ is defined to
be the set of all subsets of $\{1, 2, \dotsc, n\}$, ordered by inclusion. In a
geometric setting, the chain posets $B_n[k]$ are related to the $k$-step flag
varieties over $F_1$.

\begin{lem}
  \label{lem:boolfact}
  The Boolean algebra $B_n \cong T_1^n$, where $T_1$ is the totally
  ordered set $\{0, 1\}$ (observe that $T_1 \cong B_1$).
\end{lem}
\begin{proof}
  We can give an interpretation of $B_n$ in terms of $T_1^n$ as follows. By
  definition, every element $S \in B_n$ is some subset of
  $\{1, 2, \dotsc, n\}$, while in $T_1^n$ the elements are $n$-tuples
  consisting of $0$s and $1$s. Thus, a natural correspondence is to use $1$ to
  indicate the presence of a certain natural number in $S$ and $0$ to indicate
  its absence. For example, the elements of $B_2$ correspond to $T_1^2$ as
  follows:
  \begin{align*}
    \{1, 2\} \leftrightarrow (1, 1) \\
    \{1\} \leftrightarrow (1, 0) \\
    \{2\} \leftrightarrow (0, 1) \\
    \varnothing \leftrightarrow (0, 0).
  \end{align*}
  The bijection and order-preservation properties follow immediately.
\end{proof}

In light of~\corref{cor:dirprod} and~\thmref{thm:product}, we can now reduce
the problem of $B_n[k]$ to that of $T_1[k]$, which is considerably simpler.

\begin{lem}
  \label{lem:totalchain}
  For any $k \in \N$, $T_1[k] \cong T_k$, where $T_k$ is the totally ordered
  set $\{0, 1, \dotsc, k\}$.
\end{lem}
\begin{proof}
  The following is the Hasse diagram of $T_1[3]$, where e.g.~$011$ corresponds
  to the chain $0 \leq 1 \leq 1$:
  \begin{center}
    \begin{tikzpicture}
      \node[elem] (111) at (0, 3) {$111$};
      \node[elem] (011) at (0, 2) {$011$};
      \node[elem] (001) at (0, 1) {$001$};
      \node[elem] (000) at (0, 0) {$000$};

      \path[draw]
      (111) edge node {} (011)
      (011) edge node {} (001)
      (001) edge node {} (000);
    \end{tikzpicture}
  \end{center}
  This is clearly isomorphic to $T_3$; the general case can be seen in the same
  way.
\end{proof}

\begin{thm}
  \label{thm:boolean}
  The poset $B_n[k]$ is rank-log concave and strongly Sperner for $n, k \in \N$
  (in particular, it is rank-unimodal and Sperner).
\end{thm}
\begin{proof}
  Using \corref{cor:dirprod}, \lemref{lem:totalchain}, and
  \lemref{lem:boolfact}, we have
  \[
    B_n[k] \cong (T_1^n)[k] \cong {(T_1[k])}^n \cong T_k^n.
  \]
  It is trivial that $T_k$ is normal and rank-log concave; thus, by induction
  on \thmref{thm:product}, $T_k^n$ (and consequently $B_n[k]$) is normal and
  rank-log concave.
\end{proof}

\subsection{The poset of isotropic flags}
Closely related to the Boolean algebra explored above is the poset of isotropic
flags $I_n$, which we now define.

\begin{definition}
  Denote by $\llb n \rrb$ the set $\{1, 2, \dotsc, n\}$ and, likewise,
  $\llb n' \rrb = \{1', 2', \dotsc, n'\}$. The \emph{poset of isotropic
    flags} $I_n$ is the set of all subsets of $\llb n \rrb \sqcup \llb n' \rrb$
  which contain no pair $\{i, i'\}$ for any $i = 1, 2, \dotsc, n$, ordered
  by inclusion. In the simplest case, $I_1$ is shown below:
  
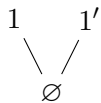
\begin{center}
    \begin{tikzpicture}
      \node[elem] (1) at (-0.5, 1) {$1$};
      \node[elem] (1') at (0.5, 1) {$1'$};
      \node[elem] (0) at (0, 0) {$\varnothing$};

      \path[draw]
      (1) edge node {} (0)
      (1') edge node {} (0);
    \end{tikzpicture}
  \end{center}
  This construction has an analogue in a geometric setting which motivates the
  choice of the name ``isotropic flags''.
\end{definition}

As with $B_n$, there exists a ``factorization'' of $I_n$ as a direct product of
simpler parts.

\begin{lem}
  \label{lem:iffact}
  For $n \in \N$, we have $I_n \cong I_1^n$.
\end{lem}
\begin{proof}
  Similarly to \lemref{lem:boolfact}, there is a natural correspondence between
  the elements of $I_n$ and those of $I_1^n$: since we enforce the condition
  that there are no pairs $\{i, i'\}$ in any element of $I_n$, we use $1$ in
  position $i$ of a tuple in $I_1^n$ to denote the presence of $i$, $1'$ to
  denote the presence of $i'$, and $0$ to denote the absence of both. For
  example, the element $\{1, 3'\}\in I_3$ would correspond to the tuple
  $(1, 0, 1') \in I_1^3$. That this indeed gives an isomorphism is easily seen
  in the same way as it was in \lemref{lem:boolfact}.
\end{proof}

\begin{lem}
  \label{lem:ifprop}
  For $k \in \N$, $I_1[k]$ is normal and rank-log concave.
\end{lem}
\begin{proof}
  As in the proof of \lemref{lem:totalchain}, we will give $I_1[3]$ as an
  example and let the general case follow similarly (as before, $01'1'$ is
  shorthand for $0 \leq 1' \leq 1'$, etc.):
  
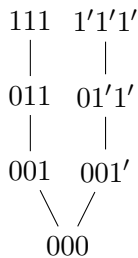
\begin{center}
    \begin{tikzpicture}
      \node[elem] (111) at (-0.5, 3) {$111$};
      \node[elem] (1'1'1') at (0.5, 3) {$1'1'1'$};

      \node[elem] (011) at (-0.5, 2) {$011$};
      \node[elem] (01'1') at (0.5, 2) {$01'1'$};

      \node[elem] (001) at (-0.5, 1) {$001$};
      \node[elem] (001') at (0.5, 1) {$001'$};

      \node[elem] (000) at (0, 0) {$000$};

      \path[draw]
      (111) edge node {} (011)
      (1'1'1') edge node {} (01'1')
      (011) edge node {} (001)
      (01'1') edge node {} (001')
      (001) edge node {} (000)
      (001') edge node {} (000);
    \end{tikzpicture}
  \end{center}
  For higher values of $k$ in $I_1[k]$, it can be seen that we simply add
  another level of two elements on top of the Hasse diagram for $I_1[k-1]$.
  Thus, normality and rank-log unimodality follow by inspection.
\end{proof}

The main theorem of this section, along with its proof, is analogous
to~\thmref{thm:boolean}.

\begin{thm}
  \label{thm:ifn}
  For $n, k \in \N$, $I_n[k]$ is strongly Sperner and rank-log concave.
\end{thm}
\begin{proof}
  By \lemref{lem:iffact} and \corref{cor:dirprod},
  \[
    I_n[k] \cong (I_1^n)[k] \cong {(I_1[k])}^n.
  \]
  Since, by \lemref{lem:ifprop}, $I_1[k]$ is normal and rank-log concave,
  \thmref{thm:product} implies that ${(I_1[k])}^n$ is also normal and rank-log
  concave.
\end{proof}

\begin{rem}
  The definition of $I_n$ generalizes naturally to higher numbers of sets
  $\llb n \rrb$, $\llb n' \rrb$, $\llb n'' \rrb$, etc.\ disallowing any
  pairwise ``matches'' between the sets (i.e.\ for any $i$, we cannot have the
  pairs $\{i, i'\}$, $\{i, i''\}$, $\{i', i''\}$). The $k$-chain posets of
  these generalizations are also normal and rank-log concave, using the same
  method of proof as for $I_n$.
\end{rem}

\end{document}